\documentclass[11pt,reqno,a4paper,preprint]{elsarticle}
\usepackage{amsmath,amssymb,amsthm}
\usepackage{mathtools}
\usepackage{enumitem}
\usepackage[margin=2.5cm,top=4cm,bottom=4cm,footskip=1cm,headsep=2cm]{geometry}
\usepackage[utf8]{inputenc}
\usepackage{float}
\usepackage{amsthm}
\usepackage[british]{babel}
\usepackage{color}

\title{On existence, uniqueness and stability of solutions to Cahn-Hilliard/Allen-Cahn systems with cross-kinetic coupling}

\def\div{\operatorname{div}}

\def\NN{\mathbb{N}}
\def\RR{\mathbb{R}}

\def\LL{\mathbf{L}}
\def\Lh{\widehat{\LL}}
\def\spn{\operatorname{span}}

\def\la{\langle}
\def\ra{\rangle}

\DeclarePairedDelimiter{\norm}{\|}{\|}
\DeclarePairedDelimiter{\snorm}{|}{|}
\newcommand{\na}{\nabla}
\def\E{\mathcal{E}}
\def\D{\mathcal{D}}

\def\na{\nabla}

\newtheorem{lemma}{Lemma}
\newtheorem{theorem}[lemma]{Theorem}
\theoremstyle{definition}
\newtheorem{definition}[lemma]{Definition}

\newtheorem{example}[lemma]{Example}

\def\dt{\partial_t}

\def\ddt{\frac{d}{dt}}

\usepackage{xcolor}

\begin{document}
\author[1]{Aaron Brunk\corref{cor1}}
\cortext[cor1]{Corresponding author}
\ead{abrunk@uni-mainz.de}
\author[2,3]{Herbert Egger}
\ead{herbert.egger@jku.at}
\author[4]{T. D. Oyedeji}
\ead{timileyin.oyedeji@tu-darmstadt.de}
\author[4]{Y. Yang}
\ead{yangyiwei.yang@mfm.tu-darmstadt.de}
\author[4]{B.-X. Xu}
\ead{xu@mfm.tu-darmstadt.de}
\affiliation[1]{organization={Institute of Mathematics, Johannes Gutenberg-University},
city={Mainz},
country={Germany}}
\affiliation[2]{organization={Johann Radon Institute for Computational and Applied Mathematics},city={Linz},country={Austria}}
\affiliation[3]{organization={Institute for Numerical Mathematics, Johannes Kepler University},city={Linz},country={Austria}}
\affiliation[4]{organization={Mechanics of Functional Materials Division, Technical University},city={Darmstadt},country={Germany}}

\begin{abstract}
A system of phase-field equations with strong-coupling through state and gradient dependent non-diagonal mobility matrices is studied. Existence of weak solutions is established by the Galerkin approximation and a-priori estimates in strong norms. Relative energy estimates are used to derive a general nonlinear stability estimate. As a consequence, a weak-strong uniqueness principle is obtained and stability with respect to model parameters is investigated. 
\end{abstract}

\begingroup
\def\uppercasenonmath#1{} 
\let\MakeUppercase\relax 
\maketitle

\vspace*{-1em}

\begin{quote} 
\end{quote}
\begin{quote}
{\small {\bf AMS-classification (2000):}
35B30   
35K61, 
35A01, 
35B35, 
35Q92  
}
\end{quote}

\vspace*{1em}

\endgroup


\section{Introduction}
\label{sec:intro}

Coupled systems of phase-field equations involving conserved and non-conserved quantities have been used by Cahn and Novick-Cohen in \cite{cahn1994} for modelling simultaneous phase separation and ordering in binary alloys. 
%
%
Similar models have been applied recently for modelling phase transformations in solid-state sintering \cite{Oyedeji2022} and, more generally, in the context of grain boundary segregation \cite{ABDELJAWAD2017528}. 
In this work we study a system with \emph{cross-kinetic coupling} of the form 
\begin{alignat}{2}
\dt\rho &= \div(\LL_{11}\nabla\mu_\rho + \LL_{12}\mu_\eta), \qquad & \mu_\rho &= -\gamma_\rho\Delta\rho + 
\partial_\rho f,
\label{eq:s1}\\
\dt\eta &= - \LL_{21} \nabla\mu_\rho - \LL_{22}\mu_\eta, \qquad & \mu_\eta &= -\gamma_\eta\Delta\eta + 
\partial_\eta f.
\label{eq:s2}
\end{alignat}
Here $\rho$, $\eta$ are the conserved and non-conserved phase-field variable, respectively. $\mu_\rho$, $\mu_\eta$ are corresponding generalized chemical potentials, $f=f(\rho,\eta)$ is an internal energy density, whose minima characterizes the thermodynamically favorable states of the systems. Furthermore, $\gamma_\rho$, $\gamma_\eta$ are the gradient parameters, and $\LL_{ij}$ are generalized mobilities or diffusivities. 
Under some general assumptions on the latter, the system describes a relaxation phenomenon accompanied by decay of the free energy 
\begin{equation}
\mathcal{E}(\rho,\eta) 
= \int_\Omega \tfrac{\gamma_\rho}{2}\snorm{\nabla\rho}^2  + \tfrac{\gamma_\eta}{2}\snorm{\nabla\eta}^2 + f(\rho,\eta) \, dx.    \label{eq:sE}
\end{equation}
Our motivation for studying systems with non-diagonal mobilities stems from asymptotic considerations for related models for phase separation in binary alloys 
\cite{Boussinot2013,Brener2012,Deng2012APF}, 
in which the presence of cross-kinetic coupling was shown essential to avoid spurious trapping effects \cite{BolladaEtAl2018} and to obtain quantitative agreement with the corresponding sharp interface limits. 
%

%
\subsection*{Related results.}
The Cahn-Hilliard/Allen-Cahn system with diagonal mobilities introduced in \cite{cahn1994} has been studied intensively in the literature. 
For constant diagonal mobility matrices and polynomial potential existence of unique global-in-time weak solutions and global attractors are proven by Brochet at al. in \cite{BROCHET199483}. In one space dimension existence with degenerate mobilities and logarithmic potential is discussed by Dal~Passo et al. in \cite{DalPasso1999}. Recently, Miranville and colleagues proved existence of unique global solutions with singular potentials and constant mobilities \cite{Miranville2019}. 
A model involving the equations of elasticity has been considered by Blesgen and Schlömerkemper~\cite{Blesgen2013}. 
The sharp interface limit has been studied  using inner and outer expansions in \cite{Nurnberg,cahn1996limiting,NOVICKCOHEN20001}. 
Some numerical investigations can be found in \cite{Yang,Huang,Xia} for the case of constant mobilities and in Barrett and Blowey~\cite{Barrett} for the case of degenerate mobilities. 
%

%
Coupled systems of multiple Cahn-Hilliard or Allen-Cahn equations with non-diagonal mobility matrices where also investigated in the literature.
Elliott et al. \cite{elliott1991} considered multi-component Cahn-Hilliard systems with constant mobility matrix and logarithmic potential and showed existence of a unique weak solution. These results were extended by Garcke et al. \cite{garcke1999} to the case of degenerated mobility matrices. 
Recently, Ehrlacher and co-workes \cite{EHRLACHER} extended the results to positive-semi definite gradient coupling. 
For coupled systems of Allen-Cahn equations, Harris et al. \cite{Harris} proved existence of a unique weak solution with constant mobility and polynomial type potential. 

\subsection*{Challenges and contributions.}
The consideration of kinetic cross-coupling in the Cahn-Hilliard/Allen-Cahn systems has some peculiarities: 
In principle, the system falls into the class of cross-diffusion systems, see e.g. \cite{Jungel2016}, but standard arguments for their analysis do not apply directly for various reasons. 
First, the system involves a second and a fourth order parabolic equation, and the cross-coupling terms involve gradients and scalar quantities. 
Second, no maximum principle for the Allen-Cahn component is available due to the cross-coupling terms. 
Furthermore, a third order derivative of the Allen-Cahn component appears in the Cahn-Hilliard equation, which implies a strong coupling of the equations. 
In order to avoid working with higher order derivatives, we use auxiliary variables $\mu_\rho$, $\mu_\eta$ corresponding to the variational derivative of the energy functional with respect to the phase-field variables $\rho$ and $\eta$, respectively.

The main focus of the current manuscript is a detailed analysis of the model \eqref{eq:s1}--\eqref{eq:s2}. 
We will establish existence of a dissipative weak solution under rather general conditions on the model parameters, in particular allowing the components of the mobility matrix to depend on the phase-field variables as well as their gradients. 
In addition, we present a nonlinear stability analysis allowing us to prove a weak-strong uniqueness result and stability w.r.t. perturbations in the model parameters.  
A main ingredient for the latter results are \emph{relative energy estimates}, which are a well-known tool for the analysis of hyperbolic conservation laws; see \cite{Dafermos.1979b,Dafermos.1979}. More recently, they were applied in the context of fluid dynamics for compressible Navier-Stokes equations and related systems; see e.g. \cite{Emmrich.2018,Feireisl.2018,Feireisl.2019,Giesselmann.2017,Hosek.2019}. 
In \cite{Brunkphd,brunkp}, we utilized relative energy estimates to provide a thorough analysis for the Cahn-Hilliard equation with concentration dependent mobilities and more general models for spinodal decomposition and viscoelastic phase separation. 
\subsection*{Outline.}
In Section \ref{sec:prelim}, we introduce our notation, main assumptions, and collect some preliminary results. 
In Section~\ref{sec:main}, we introduce our main results. Their proofs are presented in Sections~\ref{sec:existence}--\ref{sec:perturbation}. 
%
%
For illustration, some numerical tests are presented in Section~\ref{sec:num}.

\section{Notation and preliminaries}
\label{sec:prelim}

The system \eqref{eq:s1}--\eqref{eq:s2} will be investigated on a finite time interval $(0,T)$. 
To avoid the introduction of boundary conditions, we consider a spatially periodic setting, i.e., 
\begin{itemize}
\item[(A0)] $\Omega \subset \RR^d$, $d=2,3$ is a square/cube and identified with the $d$-dimensional torus $\mathcal{T}^d$.\\
Functions on $\Omega$ are assumed to be periodic throughout the paper. 
\end{itemize}
With minor changes, all results derived in the paper carry over to sufficiently regular bounded domains and more general boundary conditions.
By $L^p(\Omega)$, $W^{k,p}(\Omega)$, we denote the Lebesgue and Sobolev spaces of periodic functions with norms $\norm{\cdot}_{0,p}$ and $\norm{\cdot}_{k,p}$.
We abbreviate $H^k(\Omega)=W^{k,2}(\Omega)$ and $\norm{\cdot}_{k} = \norm{\cdot}_{k,2}$. 
The corresponding dual spaces are denoted by $H^{-k}(\Omega)=H^k(\Omega)'$, with the dual norms defined by
\begin{align} \label{eq:dualnorm}
    \norm{r}_{-k} = \sup_{v \in H^k(\Omega)} \frac{\la r, v\ra}{\|v\|_{k}}.
\end{align}
The symbol $\langle \cdot, \cdot\rangle$ denotes the duality product on $H^{-k}(\Omega) \times H^k(\Omega)$, and the same symbol is also used for the scalar product on $L^2(\Omega)$, which is defined by
\begin{align*}
\la u, v \ra = \int_\Omega u \cdot v \, dx \qquad \forall u,v \in L^2(\Omega).    
\end{align*}
By $L^2_0(\Omega) \subset L^2(\Omega)$, we denote the space of square integrable functions with zero average.
We write $L^p(a,b;X)$, $W^{k,p}(a,b;X)$, and
$H^k(a,b;X)$ for the Bochner spaces of integrable or differentiable functions on the time interval $(a,b)$ with values in a Banach space $X$. The corresponding norms are denoted by $\|\cdot\|_{L^p(X)}$,  $\|\cdot\|_{H^k(X)}$, etc. 
$C([a,b];X)$ denotes the corresponding space of continuous functions of time with values in $X$.

\subsection*{Energy dissipation}

We now give an informal derivation of an energy dissipation identity, which motivates the basic assumptions on our model parameters stated below. 
As a starting point, let us note that $\mu_\rho$, $\mu_\eta$ correspond to the variational derivatives of the energy functional \eqref{eq:sE}.
By formal differentiation of the energy along a spatially periodic solution of \eqref{eq:s1}--\eqref{eq:s2}, we then see that
\begin{align}\label{eq:defnD}
\tfrac{d}{dt} \mathcal{E}(\rho,\eta) 
&= \la \dt\rho ,\mu_\rho \ra + \la \dt\eta ,\mu_\eta \ra \\
&= - \la \LL_{11}\nabla\mu_\rho + \LL_{12} \mu_\eta,\nabla \mu_\rho \ra -  \la \LL_{21} \nabla\mu_\rho + \LL_{22} \mu_\eta, \mu_\eta\ra \notag \\
&= - \int_\Omega \begin{pmatrix} \nabla\mu_\rho \\ \mu_\eta\end{pmatrix}^\top \cdot\LL \cdot\begin{pmatrix} \nabla\mu_\rho \notag\\ \mu_\eta\end{pmatrix} =: -\D_{\rho,\eta}(\mu_\rho,\mu_\eta).
\end{align}
For the second step, we used the model equations and integration-by-parts. 
Let us note that the mobility matrix $\LL$, and hence also the dissipation functional $\D$, in general may depend on the phase-field variables $\rho$, $\eta$ and their gradients, which is indicated by the subscripts. 
In order to ensure energy dissipation, and hence thermodynamic consistency of the model, one should require that the matrix $\LL$ is positive (semi-)definite. 

\subsection*{Main assumptions}
To prove existence of a sufficiently regular solution and to guarantee thermodynamic consistency, we make the following assumptions on the model parameters.
\begin{itemize}
    \item[(A1)] The interface parameters $\gamma_\rho,\gamma_\eta>0$ are positive constants.
    \item[(A2)] The mobility function $\LL: (\rho,\nabla\rho,\eta,\nabla\eta) \to \LL(\rho,\nabla\rho,\eta,\nabla\eta)$ is smooth and uniformly positive definite, i.e.,  such that 
    \begin{equation*}
      \lambda_1\snorm{\xi}^2 \geq \mathbf{\xi}^\top\LL(\cdot) \, \mathbf{\xi} \geq \lambda_0\snorm{\xi}^2 \qquad \forall \mathbf{\xi}\in\mathbb{R}^{d+1}
    \end{equation*}
    and for all arguments $(\rho,\nabla\rho,\eta,\nabla\eta) \in \RR^{2d+2}$ with positive constants $\lambda_0,\lambda_1>0$. This in particular implies that all components of $\LL$ are uniformly bounded by a constant.
    Furthermore the derivatives $|\frac{\partial \LL}{\partial \rho}|,|\frac{\partial \LL}{\partial \eta}|,|\frac{\partial \LL}{\partial \nabla\rho}|,|\frac{\partial \LL}{\partial \nabla\eta}|\leq
    C_L$ are bounded uniformly.
    \item[(A3)] The potential $f : (\rho,\eta) \to f(\rho,\eta) \ge 0$ is smooth and non-negative. Moreover
    \begin{align*}
         \snorm*{f(\rho,\eta)}&\leq p_4(\rho,\eta),\qquad \snorm*{\partial_\rho^k\partial_\eta^l f(\rho,\eta)} \leq p_{4-k-l}(\rho,\eta), \quad \forall k+l\leq 3.
    \end{align*} 
    where $p_j(\rho,\eta)$ are polynomials of maximal degree $j$ in $\rho$ and $\eta$.
    \item[(A4)] There exists a constant $\alpha\geq 0$ such that
    \begin{equation*}
        g_\alpha(\rho,\eta):= f(\rho,\eta) + \frac{\alpha}{2}|\rho|^2 + \frac{\alpha}{2}|\eta|^2 
        \qquad \text{is strictly convex.}
    \end{equation*}
\end{itemize}
%

\begin{example} \label{ex:parameters}
A typical example for the potential function $f$ and the mobility matrix $\LL$, which is based on the choices in \cite{Tonks_2015,Oyedeji2022}, is given by
\begin{align*}
 f(\rho,\eta) &= C\rho^2(1-\rho)^2 + D \, [\rho^2 + 6(1-\rho)(\eta^2 + (1-\eta)^2) \\
 &- 4(2-\rho)(\eta^3 + (1-\eta)^3)  + 3(\eta^2 + (1-\eta)^2 )^2 ]  \\  
 \intertext{with appropriate constants $C,D>0$, and}
 \LL_{11} 
 &= a \mathbf{I} +  b\mathbf{n}_{1}(\rho)\otimes \mathbf{n}_{1}(\rho), 
 \qquad \LL_{12} = \LL_{21}^\top = c \mathbf{n}_{1}(\rho), \qquad \LL_{22} = d,   
 \end{align*}
with $a,d>0$ and $b d \ge c^2$, which ensures positive definiteness of the mobility matrix. In principle, these parameters could also depend on the variables $\rho$ and $\eta$. Here $\mathbf{n}_1(\rho):=\tfrac{\nabla\rho}{\sqrt{1+\snorm{\nabla\rho}^2}}$ denotes the regularised normal vector of $\rho$.
For $b=c=0$ the mobility matrix becomes diagonal and due to the special choice of the function $f$, which seems to be the standard in the literature \cite{WANG2006953,ahmed2013}, the two equations decouple completely. 
\end{example}

Note that in \cite{Tonks_2015,Oyedeji2022}, even more complicated forms of the mobility matrix are considered, which also depend on the normal vector of the other phase-field $\eta$.

\section{Main results}
\label{sec:main}

In the following, we summarize the main results of the paper. The proofs will be given in the subsequent sections. 
We start with clarifying the notion of a weak solution. 
\begin{definition}\label{defn:weak_sol}
For given initial data 
 $\rho_0,\eta_0\in H^1(\Omega)$.
A quadruple $(\rho,\mu_\rho,\eta,\mu_\eta)$ of functions 
\begin{align*}
    \rho &\in L^2(0,T;H^3(\Omega))\cap H^1(0,T;H^{-1}(\Omega)),& \mu_\rho &\in L^2(0,T;H^1(\Omega)), \\
    \eta &\in L^2(0,T;H^2(\Omega))\cap H^1(0,T;L^2(\Omega)),& \mu_\eta &\in L^2(0,T;L^2(\Omega)),
\end{align*}
is called a weak solution of \eqref{eq:s1}--\eqref{eq:s2}, if it satisfies the variational identities
\begin{align}
 \la \dt\rho,v_1 \ra + \la \LL_{11}\nabla\mu_\rho,\nabla v_1 \ra + \la \mu_\eta\LL_{12},\nabla v_1 \ra &= 0, \label{eq:1}\\
 \la \mu_\rho,v_2 \ra - \gamma_\rho\la \nabla\rho,\nabla v_2 \ra - \la \partial_\rho f(\rho,\eta),v_2 \ra &= 0, \label{eq:2}\\
 \la \dt\eta,w_1 \ra + \la w_1\LL_{12},\nabla\mu_\rho \ra + \la \LL_{22}\mu_\eta,w_1 \ra &= 0,\label{eq:3}\\
 \la \mu_\eta,w_2 \ra - \gamma_\eta\la \nabla\eta,\nabla w_2 \ra - \la \partial_\eta f(\rho,\eta),w_2 \ra &= 0, \label{eq:4}
\end{align}
for all test functions $v_1,v_2,w_2 \in H^1(\Omega)$ and $w_1\in L^2(\Omega)$, and a.a. $0 \le t \le T$. Furthermore, the initial data are attained in a weak sense.
Note that  the solution components depend on $t$, while the test functions are independent of time.
A weak solution is called \emph{dissipative}, if it additionally satisfies 
\begin{align}
\mathcal{E}(\rho,\eta)(t) + \int_0^t \D_{\rho,\eta}(\mu_\rho,\mu_\eta)(s) \, ds \le \E(\rho,\eta)(0). \label{eq:energy_ineq}
\end{align}
\end{definition}

By the usual interpolation theorems for Bochner spaces \cite[Theorem II.5.13]{Boyer2013}, one can deduce that the phase-field components of a weak solution also satisfy $\rho,\eta \in C([0,T];H^1(\Omega))$.

With this in mind, we can now state the first main result of our paper. 

\begin{theorem}[Existence of dissipative weak solutions]\label{thm:existence} $ $\\
Let (A0)--(A4) hold. Then for any $T>0$ and any pair of initial values $\phi_0,\eta_0\in H^1(\Omega)$, there exists at least one dissipative weak solution in the sense of Definition~\ref{defn:weak_sol}. 
\end{theorem}
A detailed proof of this result is presented in Section~\ref{sec:existence}.
%
%
As a next step, we investigate stability of weak solutions w.r.t. perturbations. Using assumption (A4), one can see that the regularized energy functional $\E_\alpha(\rho,\eta) = \E(\rho,\eta) + \frac{\alpha}{2} \|\rho\|^2_0 + \frac{\alpha}{2} \|\eta\|^2_{0}$ is strictly convex. 
In order to measure the distance between two pairs of functions $(\rho,\eta)$ and $(\hat \rho,\hat \eta)$, we then utilize the \emph{relative energy} functional
\begin{align*}
    \mathcal{E}_\alpha(\rho,\eta|\hat\rho,\hat\eta) := \int_\Omega &\frac{\gamma_\rho}{2}\snorm{\nabla(\rho-\hat\rho)}^2 + \frac{\gamma_\eta}{2}\snorm{\nabla(\eta-\hat\eta)}^2 + \frac{\alpha}{2}\snorm{\rho-\hat\rho}^2 + \frac{\alpha}{2}\snorm{\eta-\hat\eta}^2\\
    &+ f(\rho,\eta) - f(\hat\rho,\hat\eta) -\partial_\rho f(\hat\rho,\hat\eta)(\rho-\hat\rho) -\partial_\eta f(\hat\rho,\hat\eta)(\eta-\hat\eta).
\end{align*}
This corresponds to the Bregman distance induced by the strictly convex regularized energy functional $\E_\alpha(\cdot)$; see \cite{Dafermos.1979}. 
As particular candidate functions, we consider weak solutions $(\rho,\mu_\rho,\eta,\mu_\eta)$ of \eqref{eq:s1}--\eqref{eq:s2} and certain perturbations thereof. 
For any given, sufficiently smooth function $(\hat\rho,\hat\mu_\rho,\hat\eta,\hat\mu_\eta)$, we define the residuals $r_i$, $i=1,...4$ via
\begin{align}
 \la r_1,v_1 \ra &:= \la \dt\hat\rho,v_1 \ra + \la \LL_{11}\nabla\hat\mu_\rho,\nabla v_1 \ra + \la \hat\mu_\eta\LL_{12},\nabla v_1 \ra , \label{eq:p1}\\
 \la r_2,v_2 \ra &:= \la \hat\mu_\rho,v_2 \ra - \gamma_\rho\la \nabla\hat\rho,\nabla v_2 \ra - \la \partial_\rho f(\hat\rho,\hat\eta),v_2 \ra , \label{eq:p2}\\
   \la r_3,w_1 \ra &:= \la \dt\hat\eta,w_1 \ra + \la w_1\LL_{12},\nabla\hat\mu_\rho \ra + \la \LL_{22}\hat\mu_\eta,w_1 \ra,\label{eq:p3}\\
 \la r_4,w_2 \ra &:= \la \hat\mu_\eta,w_2 \ra - \gamma_\eta\la \nabla\hat\eta,\nabla w_2 \ra - \la \partial_\eta f(\hat\rho,\hat\eta),w_2 \ra. \label{eq:p4}
\end{align}
Like in Definition~\ref{defn:weak_sol}, the variational identities are assumed to hold for all test functions $v_1,v_2,w_2 \in H^1(\Omega)$ and $w_1 \in L^2(\Omega)$, and for a.a. $0 \le t \le T$.
Let us note that the mobility matrix depends on  $(\rho,\eta)$, i.e., we use some sort of linearization of the system around this solution.
We can now state the second main result of our manuscript.

\begin{theorem}[Nonlinear stability]\label{thm:stab} $ $\\
Let $(\rho,\mu_\rho,\eta,\mu_\eta)$ be a dissipative weak solution solving \eqref{eq:s1}--\eqref{eq:s2} and $r_i$, $i=1,\ldots,4$ denote the residuals defined by \eqref{eq:p1}--\eqref{eq:p4} for a set $(\hat \rho,\hat \mu_\rho,\hat \eta,\hat \mu_\eta)$ of functions satisfying
\begin{align*}
    \hat \rho &\in L^2(0,T;H^3(\Omega))\cap H^1(0,T;H^{-1}(\Omega))\cap W^{1,1}(0,T;L^2(\Omega)), &  
    \hat \mu_\rho &\in L^2(0,T;H^1(\Omega))\\
    \hat \eta &\in L^2(0,T;H^2(\Omega))\cap H^1(0,T;L^2(\Omega)), &  \hat \mu_\eta &\in L^2(0,T;L^2(\Omega)).
\end{align*}
Then the following stability estimate holds 
\begin{align} \label{eq:stability}
\E_\alpha(\rho,\eta|\hat\rho,\hat\eta)(t) &+ \int_0^t \D_{\rho,\eta}(\mu_\rho-\hat\mu_\rho,\mu_\eta-\hat\mu_\eta) \\
&\leq C_1\E_\alpha(\rho,\eta|\hat\rho,\hat\eta)(0) 
+ C_2\int_0^t \norm{r_1}_{-1}^2 + \norm{r_2}_{1}^2  + \norm{r_3}_{0}^2 + \norm{r_4}_{0}^2  \, ds   \notag
\end{align}
with $C_1,C_2$ depending only on the bounds for $\rho$, $\eta$, $\hat\rho$, $\hat\eta$ in $L^\infty(H^1)$ and for $\hat \rho,\hat\eta$ in $W^{1,1}(L^2)$.
\end{theorem}
The key steps of the proof of this result will be given in Section~\ref{sec:stab}.
Before we proceed to the verification of our main results, let us mention two important corollaries. 

\begin{theorem}[Stability w.r.t. initial values and weak-strong uniqueness] 
\label{thm:uniqueness} $ $\\
Let $(\rho,\mu_\rho,\eta,\mu_\eta)$ be a dissipative weak solution of \eqref{eq:s1}-\eqref{eq:s2} in the sense of Definition~\ref{defn:weak_sol}, and let $(\hat\rho,\hat\mu_\rho,\hat\eta,\hat\mu_\eta)$ denote another dissipative weak solution satisfying the additional regularity assumptions $\hat \rho\in W^{1,1}(0,T;L^2(\Omega))$ and $\nabla \hat \mu_\rho, \hat \mu_\eta \in L^2(0,T;L^\infty(\Omega))$.   
Then 
\begin{align*}
\E_\alpha(\rho,\eta|\hat\rho,\hat\eta)(t) + \int_0^t \mathcal{D}_{\rho,\eta}(\mu_\rho-\hat\mu_\rho,\mu_\eta-\hat\mu_\eta) ds \leq C_1\E_\alpha(\rho,\eta|\hat\rho,\hat\eta)(0). 
\end{align*}
If $(\rho(0),\eta(0))=(\hat\rho(0),\hat\eta(0))$, then $(\rho,\eta) = (\hat\rho, \hat\eta)$ and $(\mu_\rho,\mu_\eta)=(\hat\mu_\rho, \hat\mu_\eta)$ coincide.
\end{theorem}
The proof of this result will be given in Section~\ref{sec:uniqueness}. 
The second claim asserts that existence of a sufficiently regular weak solution implies the uniqueness of dissipative weak solutions, which may be referred to as a \emph{weak-strong uniqueness principle}; see e.g. \cite{Feireisl2022}.

\medskip 

As another consequence of our abstract error analysis, let us mention the stability of solutions w.r.t. perturbations in the model parameters.
For ease of presentation, we focus on perturbations in the mobility matrix.
We assume that 
\begin{equation}
 \norm{\LL(\hat\rho,\nabla\hat\rho,\hat\eta,\nabla\hat\eta)-\tilde\LL(\hat\rho,\nabla\hat\rho,\hat\eta,\nabla\hat\eta)}_0^2 \leq \varepsilon,  \label{eq:measure_error}
\end{equation}
where $(\hat \rho, \hat \mu_\rho,\hat \eta,\hat \mu_\eta)$ is a dissipative weak solution of \eqref{eq:s1}--\eqref{eq:s2} satisfying the regularity conditions of Theorem~\ref{thm:uniqueness}. 
We then obtain the following result. 

\begin{theorem}[Stability with respect to $\LL$]
\label{thm:perturbation} $ $\\
Let 
$(\rho,\mu_\rho,\eta,\mu_\eta)$ be a weak solution of \eqref{eq:s1}-\eqref{eq:s2} with mobility function $\LL(\cdot)$, and $(\hat\rho,\hat\mu_\rho,\hat\eta,\hat\mu_\eta)$ be a solution of \eqref{eq:s1}-\eqref{eq:s2} with the mobility function $\tilde\LL(\cdot)$ satisfying $\hat \rho \in W^{1,1}(0,T;L^2(\Omega))$ and $\nabla \hat \mu_\rho,\hat \mu_\eta \in L^2(0,T;L^\infty(\Omega)$, and subject to the same initial data. 
Then 
\begin{align*}
\E_\alpha(\rho,\eta|\hat\rho,\hat\eta)(t) + \int_0^t \D_{\rho,\eta}(\mu_\rho-\hat\mu_\rho,\mu_\eta-\hat\mu_\eta) ds \leq  C\varepsilon,
\end{align*}
with constant $C$ depending only on the bounds for $\rho$, $\eta$, $\hat\rho$, $\hat\eta$ in $L^\infty(H^1)$, on the bounds for $\hat \rho,\hat\eta$ in $W^{1,1}(L^2)$, and on that for $\nabla\hat\mu_\rho$ and $\hat\mu_\eta$ in $L^2(L^\infty)$.
\end{theorem}
The proof of this result again relies on the stability estimate of Theorem~\ref{thm:stab} and will be presented in Section~\ref{sec:perturbation}. 
Let us note that similar stability estimates can also be obtained for perturbations in the potential function $f(\cdot)$ and other problem data. 

%

\section{Proof of Theorem~\ref{thm:existence}}
\label{sec:existence}

We proceed with standard arguments for nonlinear parabolic equations \cite{Roubek2012}, i.e., Galerkin approximation and a-priori bounds. To handle the nonlinear terms, we use uniform a-priori bounds in strong norms and compact embeddings. 

\subsection*{Step~1: Galerkin approximation.} 
Following the standard procedure, we let  $\psi_n$, $n \in \NN_0$ denote the periodic eigenfunctions of
\begin{align*}
    -\Delta\psi_n = \lambda_n\psi_n.    
\end{align*}
They are chosen such that $\{\psi_n\}_{n=0}^\infty$ is  orthonormal in $L^2(\Omega)$ and orthogonal in $H^1(\Omega)$. 
Note that $\psi_0=const$ is one of the eigenfunctions.
%
This allows us to introduce the finite dimensional spaces $V_m=\spn\{\psi_n : 0 \le n \le m\}$ 
together with the associated projections 
%
$ P_{m} \, g = \sum\nolimits_{n=0}^m \la g,\psi_n \ra \psi_n$, which are orthogonal in $L^2(\Omega)$ and bounded in $H^k(\Omega)$, $k \ge 1$. 
We can then define the following Galerkin approximation for the weak form of \eqref{eq:s1}--\eqref{eq:s2}:
Find $\rho_m,\eta_m,\mu_{\rho,m},\mu_{\eta,m} \in H^1(0,T;V_m)$ such that
\begin{align}
 \la \dt\rho_m,v_1 \ra + \la \LL_{11,m}\nabla\mu_{\rho,m},\nabla v_1 \ra + \la \mu_{\eta,m}\LL_{12,m},\nabla v_1 \ra &= 0, \label{eq:1m}\\
 \la \mu_{\rho,m},v_2 \ra - \gamma_\rho\la \nabla\rho_m,\nabla v_2 \ra - \la \partial_\rho f(\rho_m,\eta_m),v_2 \ra &= 0, \label{eq:2m}\\
 \la \dt\eta_m,w_1 \ra + \la w_1\LL_{12,m},\nabla\mu_\rho \ra + \la \LL_{22,m}\mu_{\eta,m},w_1 \ra &= 0,\label{eq:3m}\\
 \la \mu_{\eta,m},w_2 \ra - \gamma_\eta\la \nabla\eta_m,\nabla w_2 \ra - \la \partial_\eta f(\rho_m,\eta_m),w_2 \ra &= 0, \label{eq:4m}
\end{align}
for all periodic test functions $v_1,v_2,w_1,w_2 \in V_m$ and a.a. $0 \le t \le T$, and satisfying the initial conditions $\rho_m(0)=P_{m}\rho_0$ and $\eta_m(0)=P_{m} \eta_0$. 
We write $\LL_m=\LL(\rho_m,\mu_{\rho,m},\eta_m,\mu_{\eta,m})$ to denote the evaluation of the mobility function at the approximate solutions. 

\subsection*{Step~2. Well-posedness of discrete problem.}

We can use \eqref{eq:2m} and \eqref{eq:4m} to eliminate $\mu_{\rho,m}$, $\mu_{\eta,m}$, which leads to an ordinary differential equation for $(\rho_m,\eta_m)$. 
Due to assumptions (A1)--(A3), all parameter functions are smooth w.r.t. their arguments, and existence of a unique discrete solution $\rho_m,\eta_m \in C^1([0,T_m];V_m)$ up to some time $t \le T_m$ follows by the Picard-Lindelöf theorem. By \eqref{eq:2m} and \eqref{eq:4m}, we obtain $\mu_{\rho,m},\mu_{\eta,m} \in C([0,T_m];V_m)$.

\subsection*{Step~3: A-priori bounds~1.}
In order to prove existence for all $t \le T$, we  establish uniform bounds on the approximate solutions. These are obtained by testing the discrete variational identities \eqref{eq:1m}--\eqref{eq:4m} with $v_1=\mu_{\rho,m}$, $v_2=\dt\rho_m$, $w_1=\mu_{\eta,m}$,  and $w_2=\dt\eta_m$.  
With the same reasoning as on the continuous level, we obtain
\begin{align}
\mathcal{E}(\rho_m,\eta_m)(t) + \int_0^t 
\D_{\rho_m,\eta_m}(\mu_{\rho,m},\mu_{\eta,m})(s) \, ds =  \mathcal{E}(\rho_m,\eta_m)(0) \label{eq:disc_energy}
\end{align}
for all $0 \le t \le T$. 
From (A2), we infer that $\D_{\rho_m,\eta_m}(\mu_{\rho,m},\mu_{\eta,m}) \ge \gamma \norm{\nabla\mu_{\rho,m}}_0^2 + \norm{\mu_{\eta,m}}_0^2$, and together with assumptions (A1) and (A3), we readily obtain
\begin{align*}
    \norm{\nabla\rho_m}_{L^\infty(L^2)}^2 + \norm{\nabla\eta_m}_{L^\infty(L^2)}^2 + \norm{\nabla\mu_{\rho,m}}_{L^2(L^2)}^2 + \norm{\mu_{\eta,m}}_{L^2(L^2)}^2 \leq C_0. 
\end{align*}
Here $L^p(X) = L^p(0,T_m;X)$ denotes the norms on the time interval $(0,T_m)$.
To get control on the solution in the full norms, we make the following observations:
\begin{align*}
\dt\la \rho_m,1 \ra &= 0, \qquad \qquad \la \mu_{\rho,m},1 \ra = \la \partial_\rho f(\rho_m,\eta_m), 1\ra, \\
\norm{\eta_m}^2(t)  &\leq  \norm{\eta_m}^2(0) +   C(\LL_{m})\int_0^t\norm{\nabla\mu_{\rho,m}}_0^2 +  \norm{\mu_{\eta,m}}_0^2   \, ds,
\end{align*}
which follow from the discrete variational identities \eqref{eq:1m}--\eqref{eq:4m} almost immediately. 
Using stability of the projections $P_m$, the choice of discrete initial conditions, a Poincar\'e-type inequality, and the previous estimates, we arrive at the bound
\begin{align*}
    \norm{\rho_m}_{L^\infty(H^1)}^2 + \norm{\eta_m}_{L^\infty(H^1)}^2 + \norm{\mu_{\rho,m}}_{L^2(H^1)}^2 + \norm{\mu_{\eta,m}}_{L^2(L^2)}^2 \leq C_1,
\end{align*}
which holds uniformly in the discretization parameter $m$.
As a direct consequence, we see that the discrete solution can be extended uniquely to $T_m=T$, and the a-priori bounds hold on the full time interval.

\subsection*{Step~4. A-priori bounds~2.}
From orthogonality of the basis functions $\psi_n$, the assumptions (A1)--(A4), and the variational identities \eqref{eq:1m} and \eqref{eq:3m}, we further obtain  
\begin{align*}
    \norm{\dt\rho_m}_{L^2(H^{-1})}^2 + \norm{\dt\eta_m}_{L^2(L^2)}^2 \leq C_2.
\end{align*}
By assumption (A3) and some elementary computations, one can infer that $\partial_\rho f(\rho_m,\eta_m)$ and $\partial_\eta f(\rho_m,\eta_m)$ are bounded uniformly in $L^2(0,T;H^1(\Omega))$ and $L^2(0,T;L^2(\Omega))$, respectively. From \eqref{eq:2m} and \eqref{eq:4m}, we then further deduce that $\norm{\nabla\Delta\rho_m}_{L^2(L^2)}^2 + \norm{\Delta\eta_m}_{L^2(L^2)}^2 \leq C_3$.
Elliptic regularity results for the Poisson problem
and the previous estimates then lead to 
\begin{align*}
    \norm{\rho_m}_{L^2(H^3)}^2 + \norm{\eta_m}_{L^2(H^2)}^2 \leq C_4.
\end{align*}
These stronger bounds are essential to treat the nonlinear terms in the equations later on. 

\subsection*{Step~5: Extraction of convergent sub-sequences.}
As a consequence of the a-priori bounds and the Banach-Alaoglu theorem~\cite[Theorem II.2.7]{Boyer2013}, we can choose a sub-sequence of discrete solutions, again denoted with the same index, such that
\begin{alignat*}{5}
    \rho_m &\rightharpoonup^* \bar \rho \quad && \text{in } L^\infty(0,T;H^1(\Omega)), & \qquad \qquad 
    \eta_m &\rightharpoonup^* \bar \eta \quad && \text{in } L^\infty(0,T;H^1(\Omega)), \\
    \rho_m &\rightharpoonup \bar \rho \quad && \text{in } L^2(0,T;H^3(\Omega)),& \qquad 
    \eta_m &\rightharpoonup \bar \eta \quad && \text{in } L^2(0,T;H^2(\Omega)), \\
    \dt\rho_m &\rightharpoonup \dt\bar \rho \quad && \text{in } L^2(0,T;H^{-1}(\Omega)),& \qquad 
    \dt\eta_m &\rightharpoonup \dt\bar \eta \quad &&\text{in } L^2(0,T;L^2(\Omega)), \\
    \mu_{\rho,m} &\rightharpoonup \bar \mu_{\rho} \quad && \text{in } L^2(0,T;H^1(\Omega)),& \qquad 
    \mu_{\eta,m} &\rightharpoonup \bar \mu_{\eta} \quad && \text{in } L^2(0,T;L^2(\Omega)).\\
\intertext{By application of the Aubin-Lions lemma~\cite[Theorem II.5.16]{Boyer2013}, we further deduce that}
    \rho_m &\rightarrow \bar \rho \quad && \text{in }  C([0,T];L^p(\Omega)),& \qquad 
    \eta_m &\rightarrow \bar \eta \quad && \text{in } C([0,T];L^p(\Omega)), \text{ for } p < 6,\\
    \rho_m &\rightarrow \bar \rho \quad && \text{in }  L^2(0,T;H^2(\Omega)),& \qquad 
    \eta_m &\rightarrow \bar \eta \quad && \text{in } L^2(0,T;H^1(\Omega)),
\end{alignat*}
and after extraction of another sub-sequence, we may assume w.l.o.g. that     
\begin{equation*}
    \rho_m,\nabla\rho_m,\eta_m,\nabla\eta_m \rightarrow \bar \rho,\nabla\bar \rho,\bar \eta,\nabla\bar \eta \qquad \text{pointwise a.e. in } \Omega\times(0,T).
\end{equation*}

\subsection*{Step~6: Estimates for nonlinear terms.}
Using the continuity of $\LL(\cdot)$ provided by assumption (A2), we obtain $\LL_m\rightarrow \bar \LL=\LL(\bar \rho,\bar \nabla\rho,\bar \eta,\bar \nabla\eta)$ a.e. in $\Omega\times(0,T)$. Moreover, the functions $\LL_m,\LL$ are uniformly bounded in $L^\infty(\Omega \times (0,T))$. Therefore by \cite[Exercise 8.3]{Alt2016} and since the products are uniformly bounded in $L^2(\Omega \times (0,T))$ we obtain, up to a subsequence 
\begin{alignat*}{7}
\LL_{11,m}\nabla\mu_{\rho,m} & \rightharpoonup \bar \LL_{11}\nabla\bar \mu_{\rho}, \quad & 
\LL_{12,m}\nabla\mu_{\rho,m} & \rightharpoonup \bar \LL_{12}\nabla \bar \mu_{\rho}
\qquad && \text{in } L^2(0,T;L^2(\Omega)),   \\
\LL_{12,m}\nabla\mu_{\rho,m} & \rightharpoonup \bar \LL_{12}\nabla\bar \mu_{\rho}, \quad & 
\LL_{12,m}\mu_{\eta,m} & \rightharpoonup \bar \LL_{12}\bar \mu_{\eta} \qquad && \text{in } L^2(0,T;L^2(\Omega)).
\end{alignat*}
By assumption (A3) the potential $f(\cdot)$ and its derivatives behaves like polynomials of degree less than four. Using the strong convergence of $\rho_m,\eta_m$ in $C([0,T];L^p(\Omega))$ the strong convergences of $\rho_m$ in $L^2(0,T;H^2(\Omega))$ and $\eta_h$ in $L^2(0,T;H^1(\Omega))$ allows to deduce that 
\begin{alignat*}{5}
f(\rho_m,\eta_m) & \rightarrow  f(\bar \rho,\bar \eta) \qquad && \text{in } C([0,T];L^1(\Omega)),   \\   
\partial_\rho f(\rho_m,\eta_m) & \rightarrow  \partial_\rho f(\bar \rho,\bar \eta) \qquad && \text{in } L^2(0,T;L^2(\Omega)),   \\ 
\partial_\eta f(\rho_m,\eta_m) & \rightarrow  \partial_\eta f(\bar \rho,\bar \eta)
\qquad && \text{in } L^2(0,T;L^2(\Omega)).
\end{alignat*}
Hence we obtain the necessary convergence in all nonlinear terms.

\subsection*{Step~7: Passing to the limit.}
Using the density of $\bigcup_m V_m$ in $L^2(\Omega)$, we can pass to the limit in the variational identities \eqref{eq:1m}--\eqref{eq:4m}, and verify that $(\bar \rho, \bar \eta, \bar \mu_\rho, \bar \mu_\eta)$ is a weak solution of problem \eqref{eq:s1}--\eqref{eq:s2} in the sense of Definition~\ref{defn:weak_sol}. 
To prove the energy inequality \eqref{eq:energy_ineq}, we note that the energy functional $\E(\rho,\eta)$ is quadratic in $\nabla \rho$, $\nabla \eta$, and the remaining terms are of lower order. Hence $\E(\rho,\eta)$ is weakly lower semi-continuous on $H^1(\Omega)^2$. 
Furthermore, the dissipation term can be written as
\begin{align*} 
\int_0^t
\D_{\rho,\eta}(\mu_\rho,\mu_\eta )(s)\, ds=\norm*{\LL^{1/2}\begin{pmatrix} \nabla\mu_{\rho} \\ \mu_{\eta}\end{pmatrix} }_{L^2(0,t;L^2)}^2.
\end{align*}
From the convergence results stated in Step~5 and 6, the weak lower semi-continuity of the norm, and the discrete energy identity \eqref{eq:disc_energy}, we thus obtain 
\begin{align*}
&\E(\bar\rho,\bar\eta)(t) + \int_0^t \D_{\bar \rho,\bar \eta}(\bar \mu_\rho,\bar \mu_\eta)(s) \, ds 
\le \lim\inf_m \Big\{ \E(\rho_m,\eta_m)(t) + ...\\
& \qquad ... + \int_0^t \D_{\rho_m,\eta_m}( \mu_{\rho,m},\mu_{\eta,m})(s) \, ds \Big\}
= \lim\inf_m\E(\rho_m,\eta_m)(0) = \E(\rho,\eta)(0).
\end{align*}
This shows that the weak solution $(\bar \rho,\bar \eta,\bar \mu_\rho,\bar \mu_\eta)$ constructed as limit of Galerkin approximations is also dissipative in the sense of Definition~\ref{defn:weak_sol}. 
\qed

\bigskip 

Let us remark that elliptic regularity of the underlying Poisson problem was a key step to establish the regularity results required for the handling of the nonlinear terms. 
The existence proof thus carries over almost verbatim to problems with Neumann boundary conditions on domains $\Omega$ with sufficiently smooth boundary, as treated in \cite{Elliott2000}, for instance.

\section{Proof of Theorem~\ref{thm:stab}}
\label{sec:stab}

We start with a preliminary observation concerning the relative energy functional $\E_\alpha(\cdot|\cdot)$ and the dissipation functional $\D_{\rho,\eta}(\cdot,\cdot)$, which is required in the following arguments.
\begin{lemma}\label{lem:properties}
Let (A0)--(A4) hold. Then there exist constants $C_E$, $C_D$, such that
\begin{alignat*}{5}
  \norm{\rho-\hat\rho}_1^2 + \norm{\eta-\hat\eta}_1^2  
  &\leq C_E \, \E_\alpha(\rho,\eta|\hat\rho,\hat\eta),  \\
  \norm{\mu_\rho-\hat\mu_\rho}_1^2 + \norm{\mu_\eta-\hat\mu_\eta}_0^2
  &\le C_D \, \Big( \D_{\rho,\eta}(\mu_\rho - \hat \mu_\rho,\mu_\eta - \hat \mu_\eta)  + \la \mu_\rho - \hat \mu_\rho\,1\ra^2 \Big), 
\end{alignat*}
for all functions $\rho,\hat \rho,\eta,\hat \eta \in H^1(\Omega)$, $\mu_\rho,\hat \mu_\rho \in H^1(\Omega)$ and all $\mu_\eta,\hat \mu_\eta\in L^2(\Omega)$.
\end{lemma}
\begin{proof}
These estimates are a direct consequence of the convexity of the regularized energy functional, the positive definiteness of mobility matrix, and the Poincar\'e inequality. 
\end{proof}

We now present the proof of Theorem~\ref{thm:stab} assuming that the solutions under consideration are sufficiently smooth, which simplifies the presentation of the main arguments.
By a density argument, the results however generalize to general dissipative weak solutions. 
%

\subsection*{Step~1.}
Taking the time derivative of the relative energy functional leads to
\begin{align*}
\ddt \E_\alpha(\rho,\eta|\hat\rho,\hat\eta) 
&= \gamma_\rho\la \na(\rho-\hat\rho),\na \dt(\rho-\hat\rho) \ra  + \la\partial_\rho f(\rho,\eta)-\partial_\rho f(\hat\rho,\hat\eta),\dt(\rho-\hat\rho) \ra \\
&\qquad +\gamma_\eta\la \na(\eta-\hat\eta),\na \dt(\eta-\hat\eta) \ra  + \la\partial_\eta f(\rho,\eta)-\partial_\rho f(\hat\rho,\hat\eta),\dt(\eta-\hat\eta) \ra \\
&\qquad  + \la\partial_\rho f(\rho,\eta) -\partial_\rho f(\hat\rho,\hat\eta) - \partial_{\rho\rho} f(\rho,\eta)(\rho-\hat\rho) - \partial_{\rho\eta} f(\hat\rho,\hat\eta)(\eta-\hat\eta), \dt\hat\rho\ra \\
&\qquad + \la\partial_\eta f(\rho,\eta) -\partial_\eta f(\hat\rho,\hat\eta) - \partial_{\eta\eta} f(\rho,\eta)(\eta-\hat\eta) - \partial_{\rho\eta} f(\hat\rho,\hat\eta)(\rho-\hat\rho), \dt\hat\eta\ra \\
&\qquad +\alpha\la \rho-\hat\rho, \dt(\rho-\hat\rho)\ra + \alpha\la \eta-\hat\eta, \dt(\eta-\hat\eta)\ra \\
&= (i) + (ii) + (iii) + (iv) + (v) + (vi) + (vii) + (viii).
\end{align*}
The individual terms can now be estimated separately using the variational characterization of solutions and residuals and elementary arguments.

\subsection*{Step~2.}
By inserting $v_2=\dt(\rho-\hat\rho)$ into \eqref{eq:2} and \eqref{eq:p2}, we get
\begin{align*}
(i) + (ii) 
= \la \mu_\rho-\hat\mu_\rho&+r_2,\dt(\rho-\hat\rho) \ra 
= -\la \LL_{11}\nabla(\mu_\rho-\hat\mu_\rho),\nabla(\mu_\rho-\hat\mu_\rho-r_2) \ra \\
&  - \la \LL_{12}(\mu_\eta-\hat\mu_\eta),\nabla(\mu_\rho-\hat\mu_\rho-r_2) \ra + \la r_1,\mu_\rho-\hat\mu_\rho-r_2 \ra.
\end{align*}
In the second step, we used \eqref{eq:1} and \eqref{eq:p1} with test function $v_1=\mu_\rho-\hat\mu_\rho+r_2$.
In a similar manner, we may test \eqref{eq:4} and \eqref{eq:p4} with $w_2=\dt \eta - \dt \bar \eta$, to obtain
\begin{align*}
(iii) + (iv) 
= - \la \mu_\eta-\hat\mu_\eta&+r_4, \dt (\eta - \bar \eta) \rangle 
= - \la \LL_{12}(\mu_\eta-\hat\mu_\eta+r_4),\nabla(\mu_\rho-\hat\mu_\rho) \ra \\
& - \la \LL_{22}(\mu_\eta-\hat\mu_\eta),\mu_\eta-\hat\mu_\eta+r_4 \ra  + \la r_3,\mu_\eta-\hat\mu_\eta + r_4 \ra,
\end{align*}
where we used \eqref{eq:3} and \eqref{eq:p3} with test function $w_1=\mu_\eta-\hat\mu_\eta + r_4$ in the second step.
A combination of the two expressions and elementary arguments then lead to 
\begin{align*}
(i)+(ii)+(iii)+(iv) 
\leq -(1&-\delta) \D_{\rho,\eta}(\mu_\rho-\hat\mu_\rho,\mu_\eta-\hat\mu_\eta)  + \delta \|\mu_\rho - \hat \mu_\rho\|^2_1 + \delta \|\mu_\eta - \hat \mu_\eta\|_0^2 \\
&+ C(\delta) \left( \|r_1\|_{-1}^2 + \|r_2\|_1^2 + \|r_3\|_0^2 + \|r_4\|_0^2 \right).
\end{align*}
The parameter $\delta$ stems from Young's inequalities and is still at our disposal. 

\subsection*{Step~3.}
Using \eqref{eq:2} and \eqref{eq:p2}, the mean value of $\mu_\rho - \hat \mu_\rho$ can be estimated by
\begin{align}
\la \mu_\rho-\hat\mu_\rho,1\ra^2 
&= \la\partial_\rho f(\rho,\eta) -\partial_\rho f(\hat\rho,\hat\eta) - r_2, 1\ra^2
 \leq C(f)\E_\alpha(\rho,\eta|\hat\rho,\hat\eta) + C\norm{r_2}_0^2, \label{eq:poincare_mean}
\end{align}
where we used the mean value theorem,  Sobolev embeddings, assumption (A3), and the first estimate of Lemma~\ref{lem:properties} in the second step. 
The estimate of Step~2, the second claim of Lemma~\ref{lem:properties} then lead to 
\begin{align*}
(i)+(ii)+(iii)+(iv) 
\leq -(1-(1&+C_D)\delta) \,  \D_{\rho,\eta}(\mu_\rho-\hat\mu_\rho,\mu_\eta-\hat\mu_\eta)+  c(f,\delta) \E_\alpha(\rho,\eta|\hat\rho,\hat\eta) \\
&  + C(f,\delta) \left(\norm{r_1}_{-1}^2 + \norm{r_2}_{1}^2  + \norm{r_3}_{0}^2 + \norm{r_4}_{0}^2\right).
\end{align*}
The parameter $\delta$ is still at our disposal and will be chosen later.

\subsection*{Step~4.}
By Taylor expansion, assumption (A3), and Hölder estimates, we deduce that
\begin{align*}
(v) + (vi)  
&\le C_f \, (1+\|\rho\|_{0,6} + \|\hat \rho\|_{0,6} + \|\eta\|_{0,6} + \|\hat \eta\|_{0,6}) \\
& \qquad \qquad \times (\|\rho - \hat \rho\|^2_{0,6} + \|\eta - \hat \eta\|_{0,6}^2) \, (\|\dt \hat \rho\|_0 + \|\dt \hat \eta\|_0).
\end{align*}
By Sobolev embeddings and the first claim of Lemma~\ref{lem:properties}, we arrive at
\begin{align*}
(v) + (vi)  
&\leq C_f' \, (\norm{\dt\hat\rho}_0+\norm{\dt\hat\eta}_0) \, \E_\alpha(\rho,\eta|\hat\rho,\hat\eta),
\end{align*}
with constant $C_f'$ additionally depending on the uniform bounds for $\rho,\eta,\hat\rho,\hat\eta$ in $L^\infty(H^1)$.

\subsection*{Step~5.}
With similar arguments as used in the previous steps, we can bound
\begin{align*}
(vii) + (viii) 
& \leq \delta\D_{\rho,\eta}(\mu_\rho-\hat\mu_\rho,\mu_\eta-\hat\mu_\eta) + C \, \E_\alpha(\rho,\eta|\hat\rho,\hat\eta) + \; C \left( \norm{r_1}_{-1}^2 + \norm{r_3}_{0}^2 \right).
\end{align*}
The constants $C$ in these estimates only depend on the bounds for the functions $f$ and $\LL$.

\subsection*{Step~6.}
Combination of the previous estimates and choosing $\delta=1/(2+C_D)$ leads to 
\begin{align*}
&\ddt \E_\alpha(\rho,\eta|\hat \rho,\hat \eta)  +\tfrac{1}{2} \D_{\rho,\eta}(\mu_\rho - \hat \mu_\rho,\mu_\eta - \hat \mu_\eta) \\
&\qquad \le C_1' (1 + \|\dt \hat \rho\|_0 + \|\dt \hat \eta\|_0) \, \E_\alpha(\rho,\eta|\hat \rho,\hat \eta) + C_2' \left(\norm{r_1}_{-1}^2 + \norm{r_2}_{1}^2  + \norm{r_3}_{0}^2 + \norm{r_4}_{0}^2\right),
\end{align*}
with constants $C_1',C_2'$ depending only on the 
bounds for the coefficients and uniform bounds for the phase-field components in $L^\infty(H^1)$, which are available.  
Integration in time and an application of Grönwall's lemma further yields
\begin{align*}
&\E_\alpha(\rho,\eta|\hat\rho,\hat\eta)(t) + \int_0^t \tfrac{1}{2} \D_{\rho,\eta}(\mu_\rho-\hat\mu_\rho,\mu_\eta-\hat\mu_\eta) \, ds \\
& \qquad \leq C_3'(t) \E_\alpha(\rho,\eta|\hat\rho,\hat\eta)(0) + C_4'(t) \int_0^t \norm{r_1}_{-1}^2 + \norm{r_2}_{1}^2  + \norm{r_3}_{0}^2 + \norm{r_4}_{0}^2 \, ds,
\end{align*}
with $C_3(t)$, $C_4(t)$ depending on $\int_0^t \|\dt \hat \rho\|_0 + \|\dt \hat \eta\|_0 \, ds$ and the bounds for the coefficients and solutions in the form stated above. This concludes the proof of Theorem~\ref{thm:stab}. 
\qed

\section{Proof of Theorem~\ref{thm:uniqueness}}
\label{sec:uniqueness}

By assumptions of the theorem, we denote by $(\rho,\eta,\mu_\rho,\mu_\eta)$ and $(\hat \rho,\hat \eta, \hat \mu_\rho,\hat \mu_\eta)$ two dissipative weak solutions of \eqref{eq:s1}--\eqref{eq:s2}, and additionally assume that 
\begin{align*}
 \hat \rho\in W^{1,1}(0,T;L^2(\Omega)) \qquad \text{and} \qquad \nabla \hat \mu_\rho, \hat \mu_\eta \in L^2(0,T;L^\infty(\Omega)).    
\end{align*}
Thus solution denoted by hat symbols thus has additional regularity. 
We will use Theorem~\ref{thm:stab} to proof the claims of the corollary. 
Let us start with a simple observation.
\begin{lemma}
The function $(\hat \rho,\hat \eta,\hat \mu_\rho,\hat \mu_\eta)$ satisfies \eqref{eq:p1}--\eqref{eq:p4} with residuals $r_2=r_4=0$ and
\begin{align*}
\la r_1,v_1 \ra &= \la (\LL_{11} - \Lh_{11})\nabla\hat\mu_\rho,\nabla v_1 \ra + \la \hat\mu_\eta (\LL_{12} - \Lh_{12}),\nabla v_1 \ra, \\
\la r_3,w_1 \ra &= \la (\LL_{12} - \Lh_{12})\nabla\hat\mu_\rho,w_1 \ra + \la \hat\mu_\eta(\LL_{22} - \Lh_{22}), w_1 \ra,
\end{align*}
where $\LL$, $\Lh$ result from evaluating the mobility $\LL(\cdot)$ for $(\rho,\eta)$ and $(\hat\rho,\hat\eta)$, respectively.
\end{lemma}
The claim follows immediately by subtracting the weak forms characterizing the two weak solutions. 
In the norms given in Theorem~\ref{thm:stab}, the residuals can be bounded by
\begin{align*}
 \norm{r_1}_{-1} &\leq \norm{(\LL_{11} - \Lh_{11})\nabla\hat\mu_\rho}_0 +  \norm{(\LL_{12} - \Lh_{12})\hat\mu_\eta}_0,\\
 \norm{r_3}_{0} &\leq \norm{(\LL_{12} - \Lh_{12})\nabla\hat\mu_\rho}_0 +  \norm{(\LL_{22} - \Lh_{22})\hat\mu_\eta}_0
\end{align*}
For further estimation of the individual terms, we consider the expansions
\begin{align*}
\LL -\Lh &= \LL(\rho,\nabla\rho,\eta,\nabla\eta) \mp \LL(\hat\rho,\nabla\rho,\eta,\nabla\eta)  \mp  \LL(\hat\rho,\nabla\hat\rho,\eta,\nabla\eta)\\
&\qquad \qquad \mp \LL(\hat\rho,\hat\nabla\rho,\hat\eta,\nabla\eta)  - \LL(\hat\rho,\nabla\hat\rho,\hat\eta,\nabla\hat\eta), 
\end{align*}
and by triangle inequalities and Taylor estimates, we find
\begin{align*}
|\LL -\Lh| 
&\leq C(\LL')(|\rho-\hat\rho| + |\nabla(\rho-\hat\rho)| + |\eta-\hat\eta| + |\nabla(\eta-\hat\eta)|)
\end{align*}
This allows us to estimate the two non-trivial residuals by 
\begin{align}
 \norm{r_1}_{-1}^2 &\leq \norm{\LL_{11} - \Lh_{11}}_0^2 \, \norm{\nabla\hat\mu_\rho}_{0,\infty}^2 +  \norm{\LL_{12} - \Lh_{12}}_0^2 \, \norm{\hat\mu_\eta}_{0,\infty}^2 \notag \\
&\leq C_L' \left(\norm{\nabla\hat\mu_\rho}_{0,\infty}^2 + \norm{\hat\mu_\eta}_{0,\infty}^2\right) \, \E_\alpha(\rho,\eta|\hat\rho,\hat\eta) \label{eq:resi_est1} \\
\intertext{and} 
 \norm{r_3}_{0}^2 &\leq \norm{\LL_{12} - \Lh_{12}}_0^2 \, \norm{\nabla\hat\mu_\rho}_{0,\infty}^2 +  \norm{\LL_{22} - \Lh_{22}}_0^2 \, \norm{\hat\mu_\eta}_{0,\infty}^2 \notag \\
 &\leq C_L' \left(\norm{\nabla\hat\mu_\rho}_{0,\infty}^2 + \norm{\hat\mu_\eta}_{0,\infty}^2\right) \, \E_\alpha(\rho,\eta|\hat\rho,\hat\eta) \label{eq:resi_est3}
\end{align}
From the nonlinear stability estimate of Theorem~\ref{thm:stab} and the bounds \eqref{eq:resi_est1}--\eqref{eq:resi_est3}, we see that 
\begin{align*}
\E_\alpha(\rho,\eta|\hat\rho,\hat\eta)(t) &+ \int_0^t \D_{\rho,\eta}(\mu_\rho-\hat\mu_\rho,\mu_\eta-\hat\mu_\eta) ds \\
&\leq C_1 \E_\alpha(\rho,\eta|\hat\rho,\hat\eta)(0) + C_2 \int_0^t C_3(s) \, \E_\alpha(\rho,\eta|\hat \rho,\hat \eta)(s) \, ds,
\end{align*}
with $C_3(\cdot)=2 C_L' \left(\norm{\nabla\hat\mu_\rho}_{0,\infty}^2  + \norm{\hat\mu_\eta}_{0,\infty}^2\right)$.  
The first assertion of Theorem~\ref{thm:uniqueness} now follows by a Gronwall inequality, and the second follows immediately from the first.
\qed

\section{Proof of Theorem~\ref{thm:perturbation}}
\label{sec:perturbation}
With similar reasoning as employed in the previous section, we can see that $(\hat \rho,\hat \mu_\rho,\hat \eta,\hat \mu_\eta)$ satisfies \eqref{eq:p1}--\eqref{eq:p4} with residuals $r_2=r_4=0$ and
\begin{align*}
\la r_1,v_1 \ra &= \la (\LL_{11} - \tilde\LL_{11})\nabla\hat\mu_\rho,\nabla v_1 \ra + \la \hat\mu_\eta (\LL_{12} - \tilde\LL_{12}),\nabla v_1 \ra, \\
                &= \la (\LL_{11} \mp \Lh_{11} - \tilde\LL_{11})\nabla\hat\mu_\rho,\nabla v_1 \ra + \la \hat\mu_\eta (\LL_{12} \mp \Lh_{12} - \tilde\LL_{12}),\nabla v_1 \ra\\
\la r_3,w_1 \ra &= \la (\LL_{12} - \tilde\LL_{12})\nabla\hat\mu_\rho,w_1 \ra + \la \hat\mu_\eta(\LL_{22}- \tilde\LL_{22}), w_1 \ra\\
&=\la (\LL_{12} \mp \Lh_{12}- \tilde\LL_{12})\nabla\hat\mu_\rho,w_1 \ra + \la \hat\mu_\eta(\LL_{22} \mp \Lh_{22}- \tilde\LL_{22}), w_1 \ra. 
\end{align*}
Recall that $\Lh$ results from the evaluation of the mobility $\LL(\cdot)$ for $(\hat\rho,\hat\eta)$, while $\tilde\LL$ results from the evaluation of the perturbed mobility $\tilde\LL(\cdot)$ for $(\hat\rho,\hat\eta)$.
By elementary computations and using the assumptions and previous results, one can deduce the bounds
\begin{align*}
 \norm{r_1}_{-1}^2 
 &\leq C_L \, \left(\norm{\nabla\hat\mu_\rho}_{0,\infty}^2 + \norm{\hat\mu_\eta}_{0,\infty}^2\right) \, \E_\alpha(\rho,\eta|\hat\rho,\hat\eta) + \varepsilon \, \left(\norm{\nabla\hat\mu_\rho}_{0,\infty}^2 + \norm{\hat\mu_\eta}_{0,\infty}^2\right),\notag \\
 \intertext{and}
 \norm{r_3}_{0}^2
 &\leq C_L \, \left(\norm{\nabla\hat\mu_\rho}_{0,\infty}^2 + \norm{\hat\mu_\eta}_{0,\infty}^2\right) \, \E_\alpha(\rho,\eta|\hat\rho,\hat\eta) + \varepsilon \, \left(\norm{\nabla\hat\mu_\rho}_{0,\infty}^2 + \norm{\hat\mu_\eta}_{0,\infty}^2\right),
\end{align*}
where we used the bounds \eqref{eq:measure_error} on the perturbations of the mobility functions. 
Using the abstract stability result of Theorem~\ref{thm:stab}, the above estimates for the residuals, and a Gronwall inequality then immediately yield the assertion of the theorem. 
\qed

\section{Numerical illustration}
\label{sec:num}


For illustration of our theoretical results, i.p. the stability w.r.t. perturbations in the mobility matrix, we now present some numerical tests. 
\subsection*{Problem setup}
As the spatial domain, we consider a periodic unit cell $\Omega=(0,1)^2$.
For the mobility and potential functions, we use the choices introduced in Example~\ref{ex:parameters} with $C=1$, $D=0.062$, $a=1$ and $d=1000$, while $b=c^2$ and we compare $c=0$ and $c=\sqrt{100}$. The interface parameters are set to $\gamma_\rho=\gamma_\eta=10^{-4}$. 
The steady state chemical potentials of the coupled system \eqref{eq:s1}--\eqref{eq:s2} are characterized by $\mu_{\rho}^\infty=\text{const.}$ and $\mu_\eta^\infty=0$. 
Together with the observation that $\partial_\rho f, \partial_\eta f$ do not depend on $\eta,\rho$, respectively, one can deduce that the equations determining the steady states of the phase-field variables $\rho^\infty$ and $\eta^\infty$ are independent of the parameters $a,b,c,d$.
Due to non-convexity of the free energy functional, the solutions of the steady state systems are, however, not unique and can therefore be expected to depend on the evolution, and hence also on the mobility parameter $a,b,c,d$.


\subsection*{Numerical tests}
As initial values for our computations, we choose 
\begin{align*}
  \rho_0(x,y) = \tfrac{1}{2} + \tfrac{1}{2}\sin(4\pi x)\sin(2\pi y), \quad \eta_0 = \rho_0(x,y)\chi(x\leq 0.463).
\end{align*}
For discretization we utilize a standard finite element approximation combined with a stabilized semi-implicit time-stepping scheme; see \cite{shen2010} for similar methods applied to the Cahn-Hilliard or Allen-Cahn equation. Furthermore, we validated the solutions on refined discretisations and by an Petrov-Galerkin scheme in the spirit of \cite{brunkp}.
%

\subsection*{Results}

In Figure \ref{fig:results} we display some snapshots for simulations of our test problem with $b=100$ and $b=0$.  
\begin{figure}[htbp!]
\label{fig:results}
\centering
\footnotesize
\begin{tabular}{cccc}
     \includegraphics[trim={7.5cm 4.4cm 9.5cm 4.5cm},clip,scale=0.12]{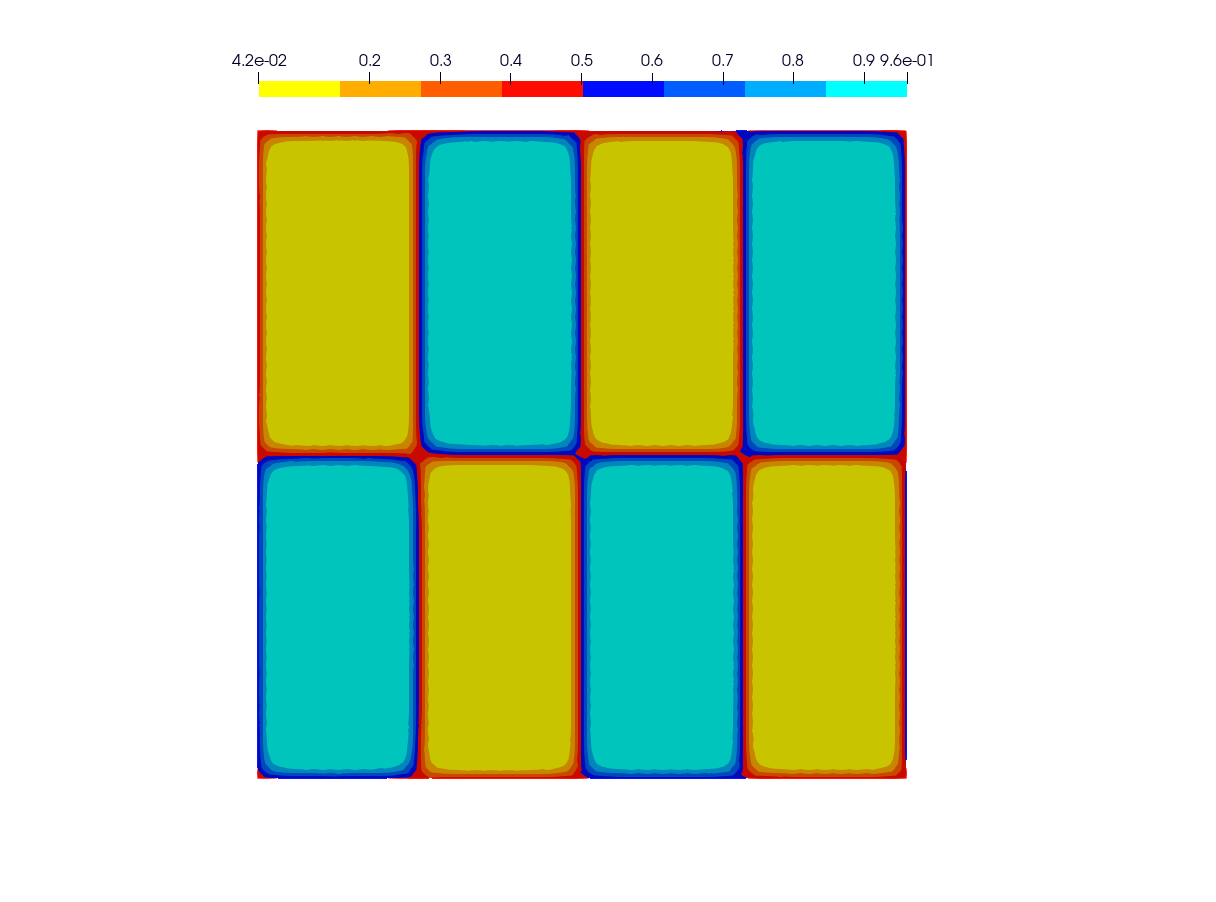} 
    &\includegraphics[trim={7.5cm 4.4cm 9.5cm 4.5cm},clip,scale=0.12]{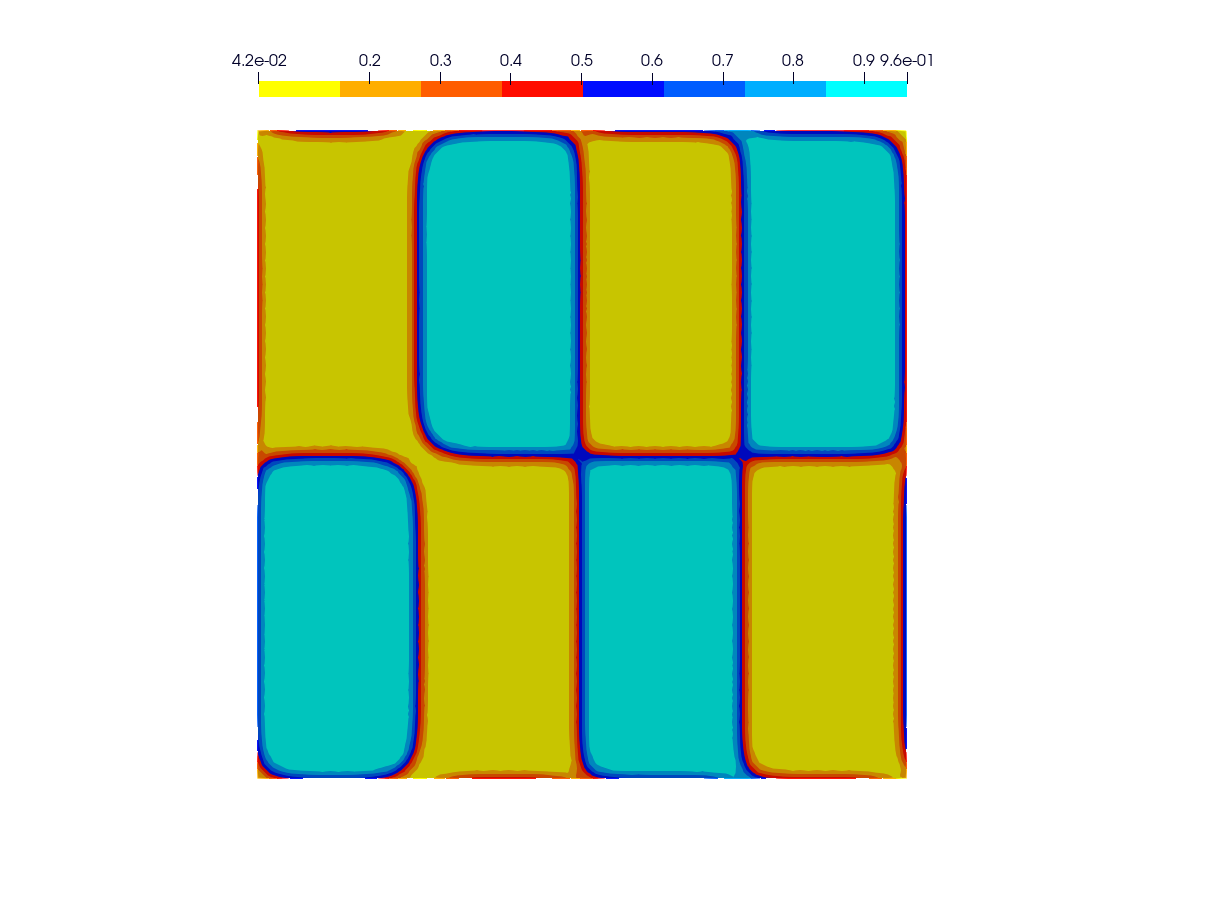} 
    &\includegraphics[trim={7.5cm 4.4cm 9.5cm 4.5cm},clip,scale=0.12]{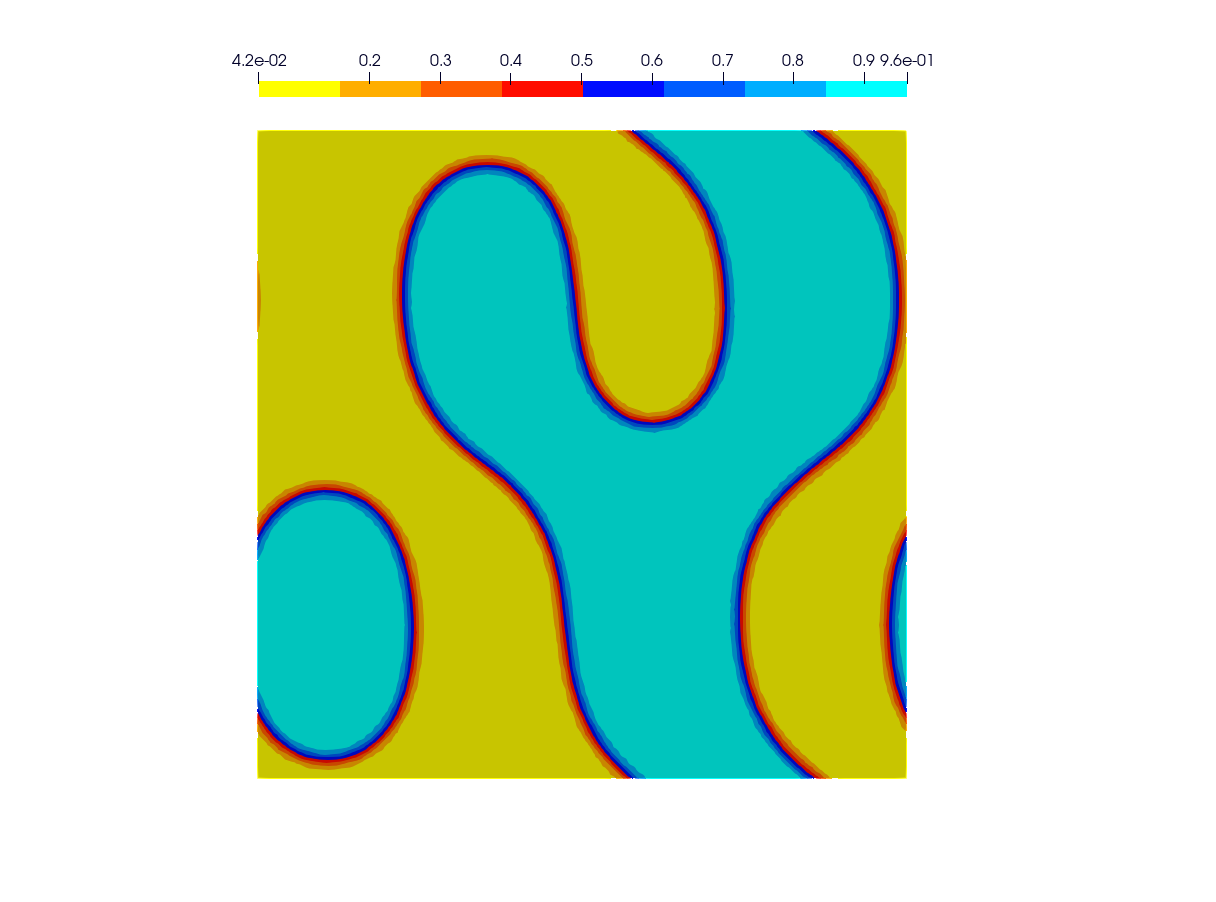} 
    &\includegraphics[trim={7.5cm 4.4cm 9.5cm 4.5cm},clip,scale=0.12]{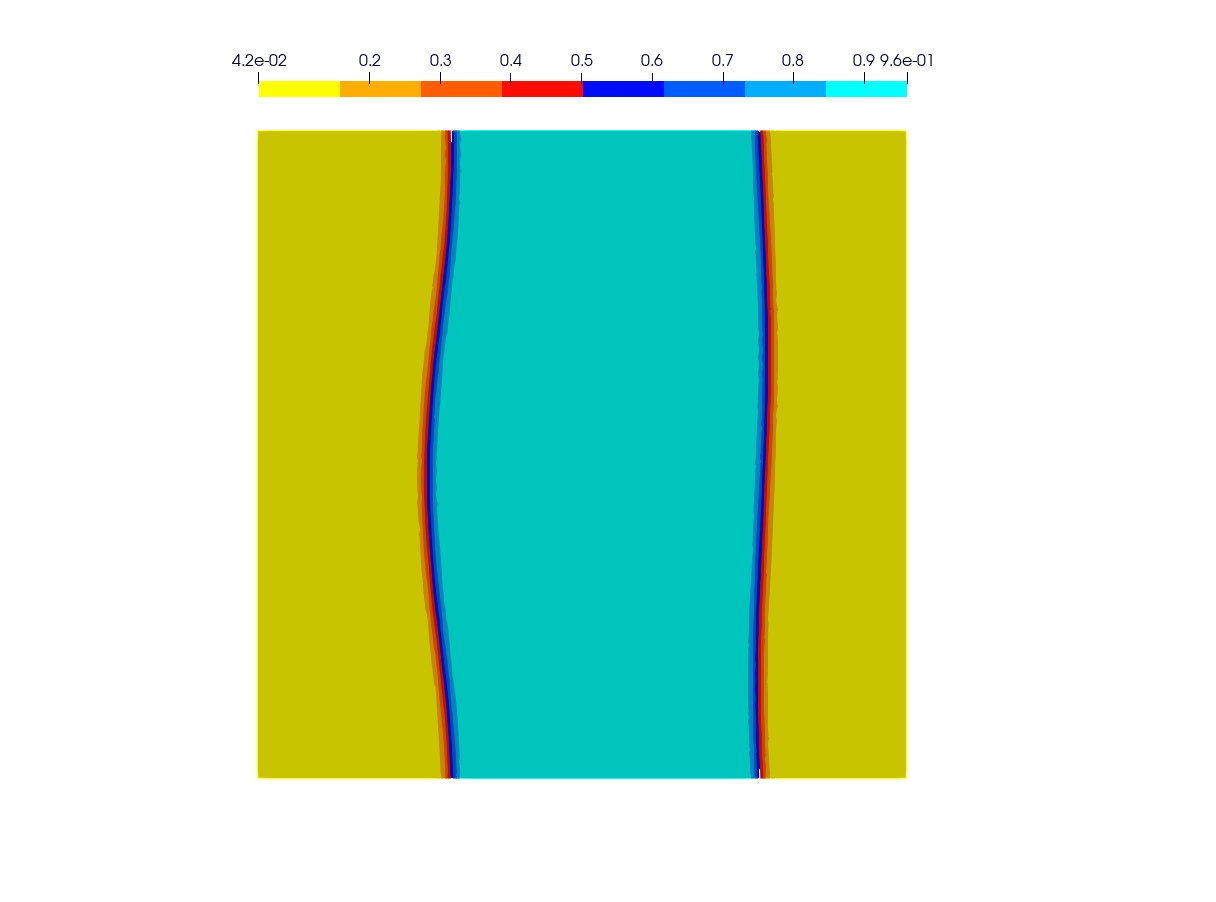} \\[1em]
    \hspace{-1em} \includegraphics[trim={7.5cm 4.4cm 9.5cm 4.5cm},clip,scale=0.123]{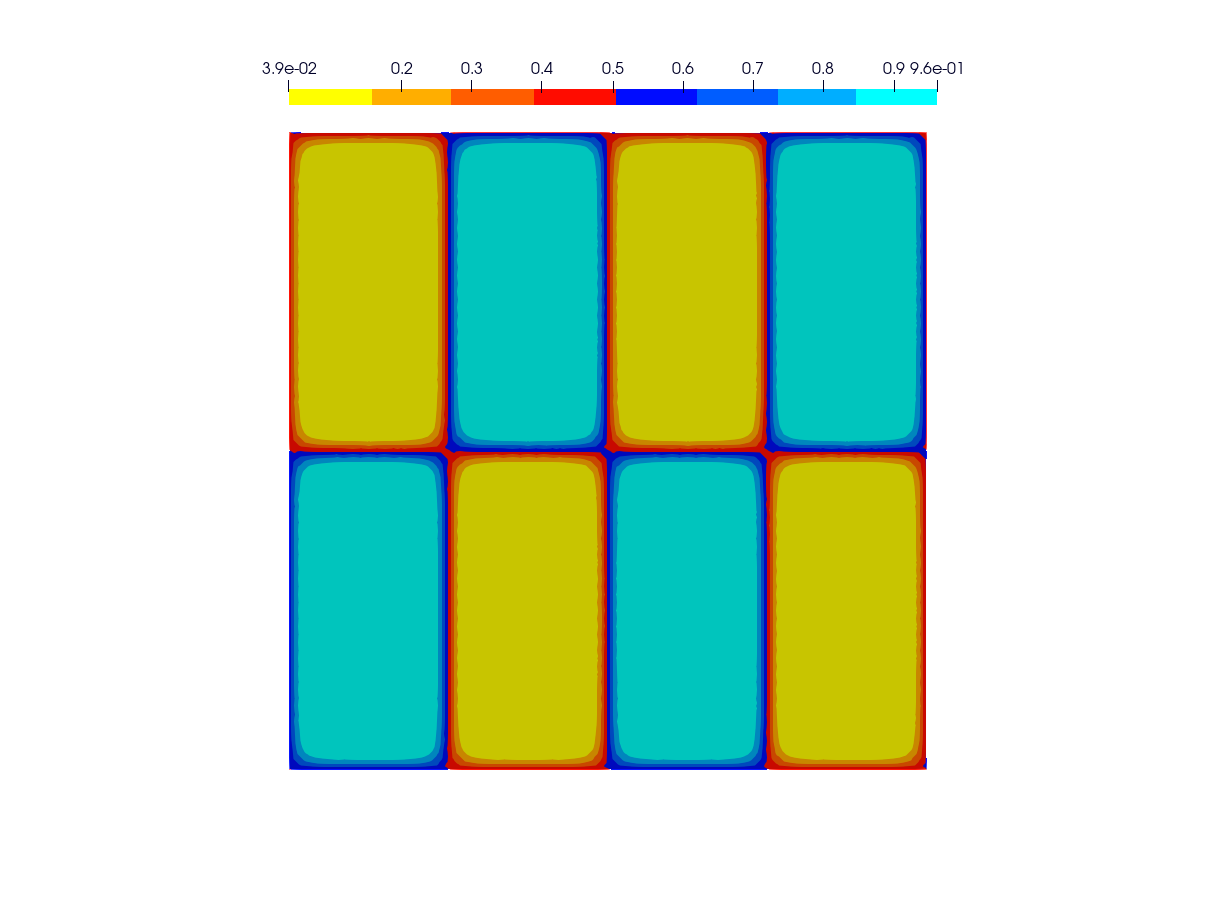} 
    &\hspace{-0.75em}\includegraphics[trim={7.5cm 4.4cm 9.5cm 4.5cm},clip,scale=0.123]{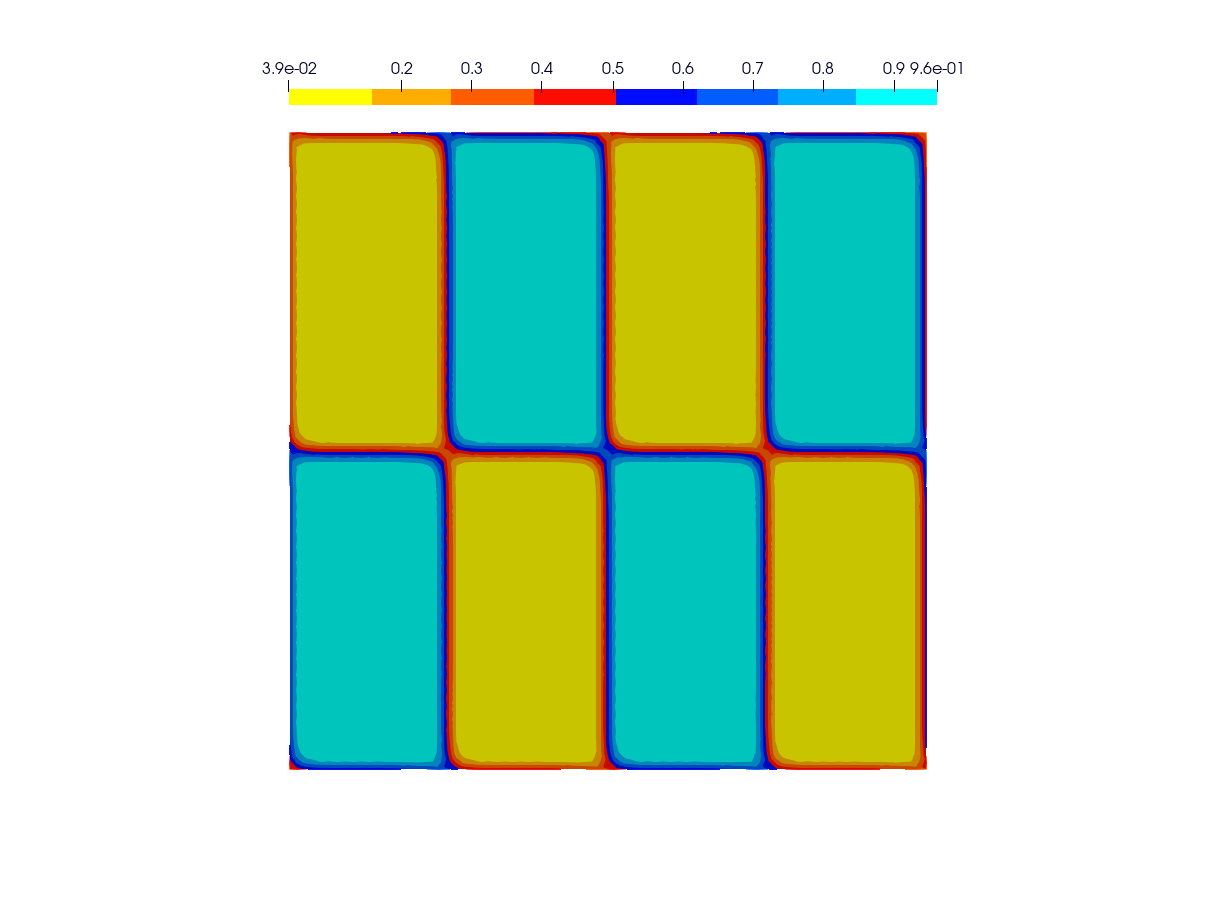} 
    &\hspace{-0.75em}\includegraphics[trim={7.5cm 4.4cm 9.5cm 4.5cm},clip,scale=0.123]{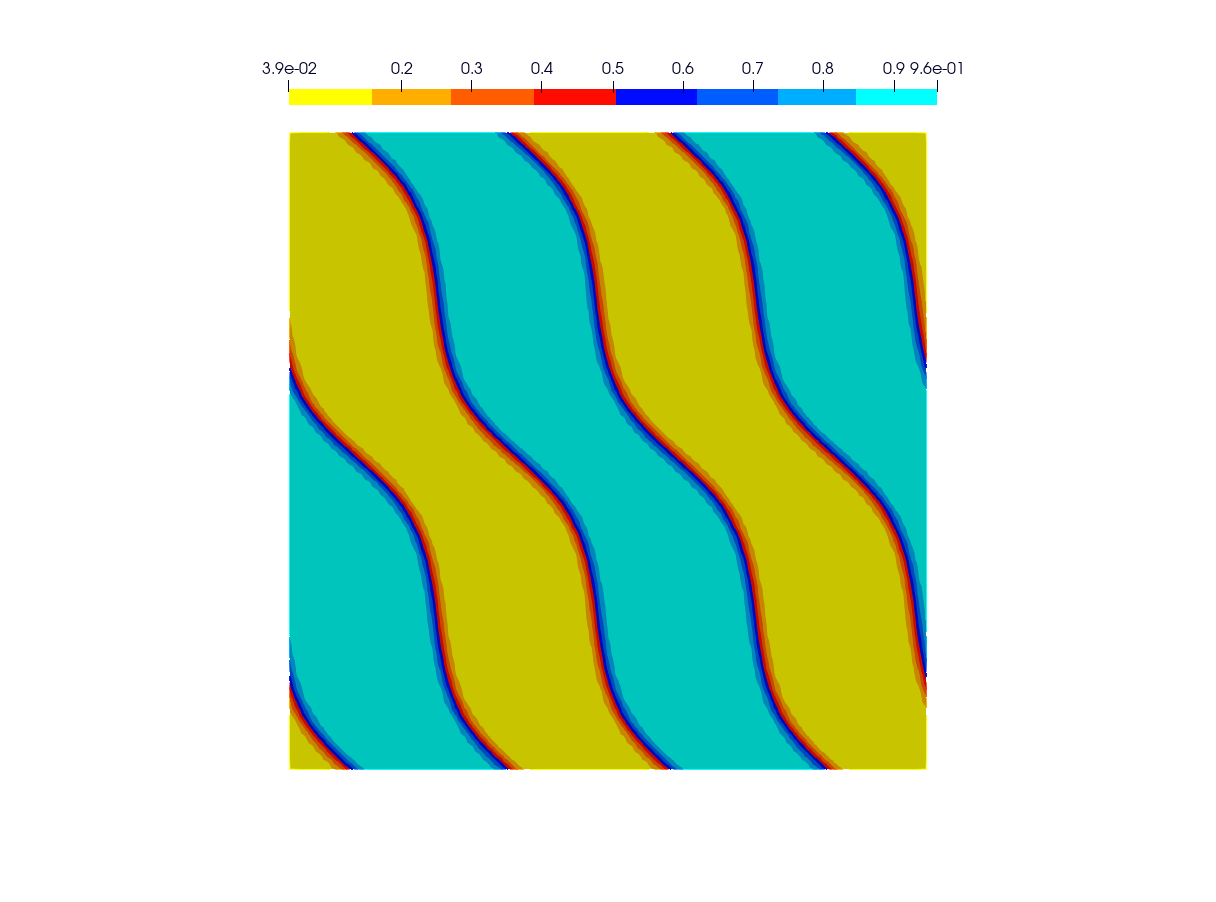} 
    &\hspace{-0.75em}\includegraphics[trim={7.5cm 4.4cm 9.5cm 4.5cm},clip,scale=0.123]{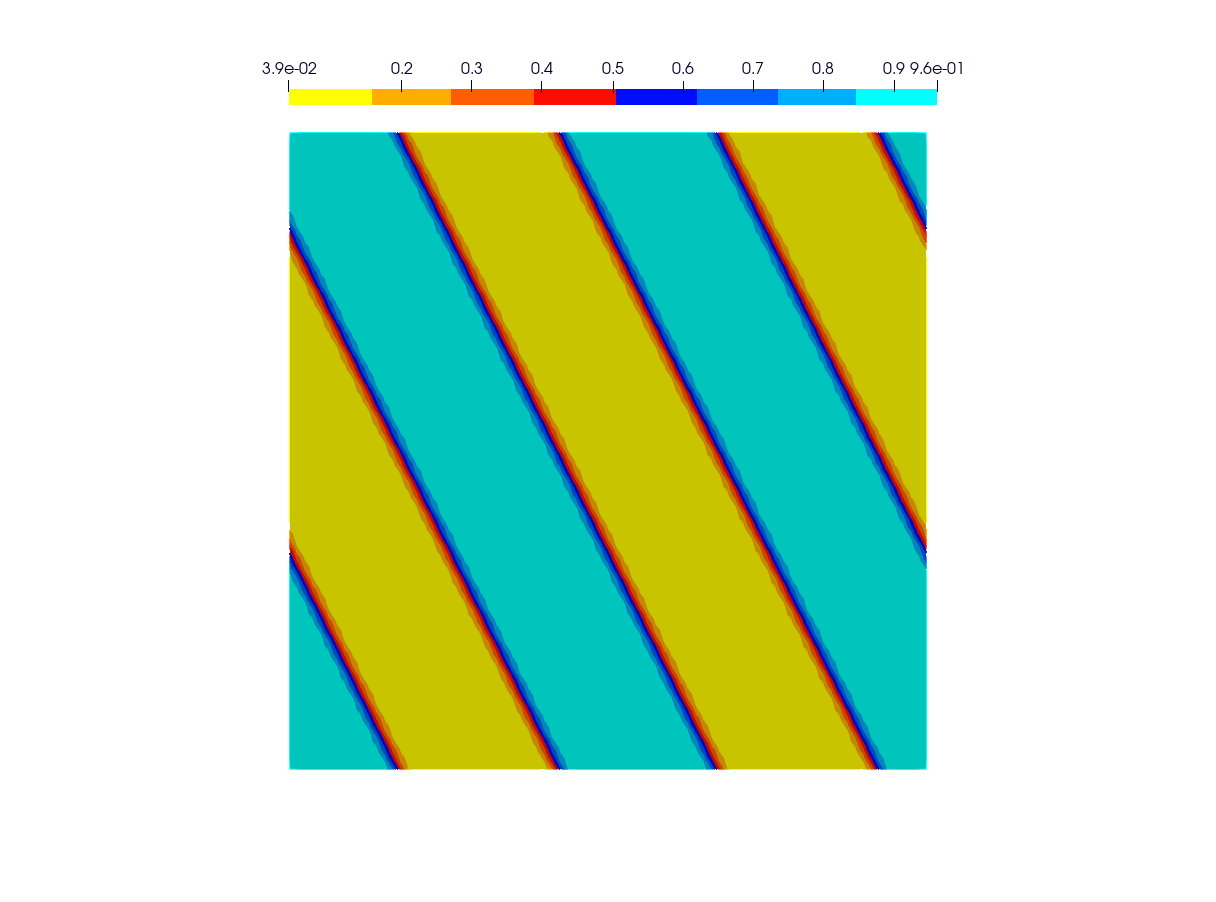} \\[-0.0em]
    t=0.1  & t=0.3 & t=2.2 & t=20\\
\end{tabular}
    \caption{Snapshots of the volume fraction $\rho$ for model with cross-kinetic coupling ($b=100$, top) and without cross-kinetic coupling ($b=0$, bottom).}
\end{figure}
In perfect agreement with our theoretical results, which predict stability w.r.t. perturbations in the model parameters but with constants growing in time, we observe very similar solutions for the two test cases for small times, but growing deviations for larger times.
To ensure that the results are independent of the discretization parameters, the tests were carried out for different choices of the mesh size $h$ and the time step $\tau$, leading to identical observations.
For time $t=20$, the two solutions for $b=0$ and $b=100$ became almost stationary but rather different, despite the fact that the steady states for the two cases are described by the same set of equations; see the discussion above. Note that for $b=0$ the strip pattern represents a quasi-stationary state of the Cahn-Hilliard equation, since both are in this uncoupled  the further evolution is prohibited. In the coupled case the additional driving force from the cross-coupling drives the Cahn-Hilliard part out of the quasi-stationary state as long as $\LL_{12}\mu_\eta$ is non-zero. By construction with the normal vector this effect is mostly localised at the interface, i.e. in this case prohibits such a quasi-stationary state. This is in perfect agreement with the interpretation of the kinetic cross-coupling term in the literature as \emph{anti-trapping term} \cite{Oyedeji2022}, i.e. an additional force which prevents trapping in quasi-stationary states.

\section{Discussion}

In this work, we studied Cahn-Hilliard/Allen-Cahn systems with non-diagonal and gradient dependent mobilities. Existence of global-in-time weak solutions was established using a-priori estimates in strong norms. Based on relative energy estimates, a nonlinear stability analysis was developed, which allowed us to prove a weak-strong uniqueness principle as well as stability estimates w.r.t perturbations in the model parameters. 
While demonstrated here for the Cahn-Hilliard/Allen-Cahn equation with cross-kinetic coupling, the basic arguments should be applicable also to more complex models, including non-isothermal extensions and the incorporation of fluid flow; see e.g. \cite{Abels2013,MARVEGGIO2021924,Francesso}.
{\footnotesize

\section*{Acknowledgement}
\begin{itemize}
    \item The authors would like to thankfully acknowledge the comments and suggestions by the two anonymous
reviewers to improve the quality of the manuscript.
\item Support by the German Science Foundation (DFG) via TRR~146 (project~C3) and SPP~2256 (project Eg-331/2-1) is gratefully acknowledged.
\end{itemize}}

\bibliographystyle{abbrv}	
\bibliography{mathconsnoncons.bib}

\end{document}